\documentclass[11pt]{amsart}
\usepackage{amsmath}
\usepackage{mathtools}
\usepackage{esint}
\usepackage{tikz}
\usepackage{comment}
\usepackage{mathrsfs}

\theoremstyle{plain}
\newtheorem{theorem}{Theorem}[section]
\newtheorem{lemma}[theorem]{Lemma}
\newtheorem{prop}[theorem]{Proposition}

\theoremstyle{definition}
\newtheorem{remark}[theorem]{Remark}
\newtheorem{definition}[theorem]{Definition}

\numberwithin{equation}{section}

\def\be{\begin{equation}}
\def\ee{\end{equation}}

\begin{document}

\title[Singular Solutions]
{Singular Solutions to the Yamabe Equation
with Prescribed Asymptotics}
\author[Qing Han]{Qing Han}
\address{Department of Mathematics\\
University of Notre Dame\\
Notre Dame, IN 46556} \email{qhan@nd.edu}
\author[Yichao Li]{Yichao Li}
\address{Department of Mathematics\\
University of Notre Dame\\
Notre Dame, IN 46556} \email{yli20@nd.edu}

\begin{abstract}
We study positive solutions of the Yamabe equation with isolated singularity and prove the 
existence of solutions with prescribed asymptotic expansions near singular points 
and an arbitrarily high order of approximation. 
\end{abstract}

\thanks{The first author acknowledges the support by the NSF
Grant DMS-1404596
and the second author acknowledges the support by the NSF grant DMS-1569162.}

\maketitle

\section{Introduction}\label{sec-introduction} 

In this paper, we study positive solutions of the Yamabe equation of the form
\begin{equation}\label{eq-Yamabe} 
\Delta u+\frac14n(n-2)u^{\frac{n+2}{n-2}}=0,\end{equation}
which are defined in the punctured ball $B_1\setminus\{0\}$, with a nonremovable singularity at the origin. 
Geometrically, for any positive solution $u$ of the equation \eqref{eq-Yamabe}, 
the corresponding conformal metric 
$g=u^{\frac4{n-2}}|dx|^2$
has a constant  scalar curvature $R_g=n(n-1)$.

In a pioneering paper \cite{CaffarelliGS1989}, Caffarelli,  Gidas,  and Spruck 
proved that positive singular solutions of \eqref{eq-Yamabe} in $B_1\setminus\{0\}$
are asymptotic to radial singular solutions of \eqref{eq-Yamabe} in $\mathbb R^n\setminus\{0\}$. 
Specifically, a positive solution $u$ of \eqref{eq-Yamabe} in $B_1\setminus \{0\}$, 
with a nonremovable singularity at the origin, 
satisfies 
\begin{equation}\label{eq-estimate-u-0}
|x|^{\frac{n-2}2}u(x)-\psi(-\ln|x|)\to 0\quad\text{as }x\to 0,\end{equation}
where $|x|^{\frac{2-n}2}\psi(-\ln|x|)$ is a positive radial solution of \eqref{eq-Yamabe} 
in $\mathbb R^n\setminus\{0\}$, with a nonremovable singularity at the origin. In fact, 
$\psi$ is a positive periodic function in $\mathbb R$. A key ingredient in the proof is a 
``measure theoretic" version of the moving plane technique.

Subsequently in \cite{KorevaarMPS1999}, Korevaar, Mazzeo, Pacard, and Schoen 
studied refined asymptotics and expanded solutions to the next order. Their result states 
that 
a positive solution $u$ of \eqref{eq-Yamabe} in $B_1\setminus \{0\}$, 
with a nonremovable singularity at the origin, 
satisfies, for some constant $\alpha>1$ and any $x\in B_{1/2}\setminus\{0\}$, 
\begin{equation}\label{eq-estimate-u-1}
\big||x|^{\frac{n-2}2}u(x)-\psi(-\ln|x|)-\phi(-\ln|x|)P_1(x)\big|\le C|x|^{\alpha},\end{equation}
where $\psi$ is the function as in \eqref{eq-estimate-u-0}, 
$P_1$ is a linear function, $\phi$ is a function given by 
$$\phi=-\psi'+\frac{n-2}2\psi.$$
The proof of \eqref{eq-estimate-u-1} is based on the classification of 
global singular solutions and analysis of linearized operators at these global singular solutions. 

There have been many results related to the theme of estimates 
\eqref{eq-estimate-u-0} and \eqref{eq-estimate-u-1}. Han, Li, and Teixeira \cite{HanLi2010} studied 
the $\sigma_k$-Yamabe equation near isolated singularities and derived similar estimates for 
its solutions. 
Caffarelli, Jin, Sire, and Xiong \cite{CaffarelliJSX2014} studied 
fractional semi-linear elliptic
equations with isolated singularities. Refer to \cite{HanLi2010} and \cite{CaffarelliJSX2014} 
for more references on these subjects. 

In \cite{HLL201?}, we studied expansions of positive singular 
solutions of the Yamabe equation \eqref{eq-Yamabe} and the $\sigma_k$-Yamabe equation in $B_1\setminus\{0\}$ 
up to  arbitrary orders. 
Let  $u$ be a positive solution of \eqref{eq-Yamabe} in $B_1\setminus\{0\}$,
with a nonremovable singularity at the origin, and $\psi$ 
be the positive periodic function as in \eqref{eq-estimate-u-0}.
We proved that $\psi$ determines a 
positive sequence $\{\mu_k\}$, strictly increasing and divergent to $\infty$, which we
call the {\it index set} associated with $\psi$. 
This index set determines the decay rate of $|x|^{\frac{n-2}2}u(x)$ in the following order: 
\begin{align*}
|x|^0,\, |x|^{\mu_1},\, (-\log|x|)|x|^{\mu_2},\, |x|^{\mu_2},\,\cdots,\,
(-\log|x|)^{k-1}|x|^{\mu_k},\,\cdots,\, |x|^{\mu_k},\, \cdots.\end{align*}
Specifically, we proved, 
for any positive integer $k$ 
and any $x\in B_{1/2}\setminus\{0\}$, 
\begin{align}\label{eq-estimate-u-k}
\Big||x|^{\frac{n-2}2}u(x)-\psi(-\ln|x|)-\sum_{i=1}^k\sum_{j=0}^{i-1} 
c_{ij}(x)|x|^{\mu_i}(-\ln|x|)^j\Big|
\le C|x|^{\mu_{k+1}}(-\ln|x|)^k,
\end{align}
where  
$c_{ij}$ is a bounded smooth function in $B_{1/2}\setminus\{0\}$, 
for each $i=1, \cdots, k$ and $j=0\cdots, i-1$. 
These functions are determined by  $\psi$, up to the kernel of the linearized Yamabe equation at $\psi$. They 
can be computed in a rather mechanical way and have two sources, 
the kernel of the linearized equation and the nonlinearity. Naturally, 
the expression 
\begin{equation}\label{eq-def-approximation-u}
|x|^{-\frac{n-2}2}\big[\psi(-\ln|x|)+\sum_{i=1}^k\sum_{j=0}^{i-1} 
c_{ij}(x)|x|^{\mu_i}(-\ln|x|)^j\big]\end{equation}
is considered as an approximate solution of \eqref{eq-Yamabe}. The estimate \eqref{eq-estimate-u-k}
simply asserts that any positive solutions of \eqref{eq-Yamabe} with a nonremovable singularity at the origin
are well approximated by an approximate solution. 

It is also important to study the converse question. Can we find solutions of \eqref{eq-Yamabe} 
near approximate solutions? Mazzeo and Pacard \cite{MazzeoP1999} 
constructed complete conformal metrics with positive scalar curvature in the complement of a finite union of 
lower dimensional submanifolds on the unit sphere. In terms of the equation \eqref{eq-Yamabe}, 
they constructed a positive solution of \eqref{eq-Yamabe} 
in the complement of finitely many points in $\mathbb R^n$, with nonremovable singularity near each of 
these points and with an appropriate order of decay near infinity. 
They started with radial solutions in sufficiently small balls around 
singular points with sufficiently small neck sizes, and extended to the rest of $\mathbb R^n$ by harmonic functions 
which decay near infinity with an appropriate order. These functions serve as global approximate solutions. 
Then, they introduced weighted H\"older spaces and proved the existence of desired solutions. 

In this paper, we study a refined local version. Let $\widehat u$ be an approximate solution of \eqref{eq-Yamabe} of certain order. 
Can we find a solution $u$ of \eqref{eq-Yamabe} which is well approximated by $\widehat u$? 
To answer this question, we need to introduce 
a correct version of approximate solutions, which are modeled after  \eqref{eq-def-approximation-u}. 
The expression \eqref{eq-def-approximation-u} indicates that approximate solutions should have a leading term 
and an \lq\lq order\,\rq\rq. 
The order can be defined as the vanishing order of the left-hand side of \eqref{eq-Yamabe} evaluated at 
the approximate solution, modular a conformal factor. 


To state the main result, it is convenient to introduce the operator
\begin{equation}\label{eq-Yamabe-operator-u} 
\mathcal M(u)=\Delta u+\frac14n(n-2)u^{\frac{n+2}{n-2}}.\end{equation}
The main result is the following theorem. 

\begin{theorem}\label{mainthm}
Let $|x|^{-\frac{n-2}2}\psi(-\ln|x|)$ be a solution of \eqref{eq-Yamabe} 
in $\mathbb R^n\setminus\{0\}$ for some positive periodic function $\psi$ on $\mathbb R$, 
$\mathcal I$ the index set associated with $\psi$, and $\mu>0$ a constant 
with $\mu\notin \mathcal I$. Suppose $\widehat u$ is a smooth function 
 in $B_1\setminus\{0\}$ satisfying 
\begin{equation}\label{eq-assumption-leading-u}
\big||x|^{\frac{n-2}2}\widehat u(x)-\psi(-\ln|x|)\big|
+|x|\big|\nabla(|x|^{\frac{n-2}2}\widehat u(x)-\psi(-\ln|x|))\big|\to 0\quad\text{as }x\to 0,\end{equation}
and 
\begin{equation}\label{eq-assumption-approximate-u}
|x|^{\frac{n+2}2}\big(|\mathcal{M}(\widehat u)(x)|+|x||\nabla(\mathcal{M}(\widehat u))(x)|\big)\le C|x|^\mu,\end{equation}
for some positive constant $C$. 
Then, there exist a constant $R\in (0,1)$ and a positive solution $u$ of \eqref{eq-Yamabe} in $B_R\setminus\{0\}$
such that, for any $x\in B_R\setminus \{0\}$, 
\begin{equation*}
|x|^{\frac{n-2}2}|u(x)-\widehat u(x)|\leq C' |x|^\mu,
\end{equation*}
where $C'$ is a positive constant.
\end{theorem} 

Theorem \ref{mainthm} asserts that if $\widehat u$ is an approximate solution to the Yamabe equation which is close 
to a radial solution up to certain orders, then $\widehat u$ is close to an actual solution up to an appropriate order. 
The index set $\mathcal I$ will be defined in \eqref{eq-def-index} and is the collection of $\{\mu_i\}$ as in 
\eqref{eq-estimate-u-k}. The assumptions \eqref{eq-assumption-leading-u} and \eqref{eq-assumption-approximate-u}
simply say that $\widehat u$ has a leading term  $|x|^{-\frac{n-2}2}\psi(-\ln|x|)$ and is an approximate solution of 
\eqref{eq-Yamabe} of order $\mu$. Theorem \ref{mainthm} asserts that there always exists a positive solution of the
Yamabe equation near an approximate solution with the correct order of approximation. 

To prove Theorem \ref{mainthm}, we study the linearized operator in appropriate weighted H\"older spaces
in a ball centered at the singular point and construct solutions which decay up to order $\mu$,
the designated order of the approximation. To this end, we first project the linearized equation into a finite dimensional subspace, 
which is spanned by spherical harmonics whose degrees are less than a certain constant depending on $\mu$ 
through {\it indicial roots} of the linearized equation, and construct desired solutions in this subspace 
by solving finitely many ordinary differential equations. In the infinite dimensional complimentary subspace, we construct 
desired solutions by the method of variation, and derive a crucial estimate by a rescaling argument, since the maximum 
principle is not applicable due to the \lq\lq incorrect\rq\rq\, sign of the coefficient of the zero order term. 
It is important that such an estimate hold in balls uniformly independent of the size of the ball. 

The method in this paper provides a general procedure to construct solutions of nonlinear equations 
from solutions of their linearized equations. In general, solutions of linearized equations are not approximate solutions 
of nonlinear equations up to a designated order. We  need to correct solutions of linearized equations first to get 
approximate solutions of nonlinear equations  and then to perturb approximate solutions to get actual solutions.


The paper is organized as follows. 
In Section \ref{sec-Linearized}, we solve the linearized Yamabe equation
in appropriate weighted H\"older spaces and derive all necessary estimates. In Section \ref{sec-approximate-solution}, 
we 
prove Theorem \ref{mainthm} by the contraction mapping principle. 

\section{Linearized Yamabe Equations}\label{sec-Linearized}

In this section, we solve the linearized Yamabe equation in appropriate weighted H\"older spaces. 
It is crucial to introduce appropriate boundary conditions to have desired estimates. 
Some of these estimates 
are not {\it a priori}, and hold only for specially constructed solutions.

We first recall some basic results concerning the linearized Yamabe equations. 
Following \cite{KorevaarMPS1999}, we express the equation \eqref{eq-Yamabe}
in the cylindrical metric $dt^2+d\theta^2$ on $\mathbb R_+\times\mathbb S^{n-1}$. Using the polar coordinates 
$x=r\theta$, with $r=|x|$ and $\theta\in\mathbb S^{n-1}$,  
and the cylindrical variable $t=-\ln r$, we set 
\begin{equation}\label{eq-def-v}v(t, \theta)=r^{\frac{n-2}2}u(r\theta).\end{equation}
A straightforward calculation  transforms the equation 
\eqref{eq-Yamabe} to 
\begin{equation}\label{eq-yamabe-v}
v_{tt}+\Delta_{\theta}v-\frac14(n-2)^2v+\frac14n(n-2)v^{\frac{n+2}{n-2}}=0.\end{equation}
In this paper, we always consider positive solutions $u$ of \eqref{eq-Yamabe} in $B_1\setminus \{0\}$, 
with a nonremovable singularity at the origin. With $v$ given by \eqref{eq-def-v}, 
we shall say that $v$ has a nonremovable singularity at infinity.

If $u$ is a radial positive solution of \eqref{eq-Yamabe} in $\mathbb R^n\setminus\{0\}$, we set 
\begin{equation}\label{eq-def-psi}\psi(t)=r^{\frac{n-2}2}u(r\theta).\end{equation}
Then, $\psi$ is a positive solution of 
\begin{equation}\label{eq-psi}
\psi''-\frac14(n-2)^2\psi+\frac14n(n-2)\psi^{\frac{n+2}{n-2}}=0\quad\text{on }\mathbb R.\end{equation}
According to \cite{CaffarelliGS1989}, any positive solution of \eqref{eq-psi} with a nonremovable singularity
at infinity is 
a smooth positive periodic function, 
either a positive constant, which is unique, or a smooth nonconstant periodic function. 

For a fixed positive periodic solution 
$\psi$ of \eqref{eq-psi}, the linearized operator of \eqref{eq-yamabe-v} at $\psi$ 
is given by, for any $w\in C^2(\mathbb R\times\mathbb S^{n-1})$, 
\begin{equation}\label{eq-linearization}
Lw=
w_{tt}+\Delta_{\theta}w-\frac14(n-2)^2w+\frac14n(n+2)\psi^{\frac{4}{n-2}}w.\end{equation}
Our main interest is its kernel. 

Let $\{\lambda_i\}$ be the sequence of eigenvalues of $-\Delta_{\theta}$ on $\mathbb S^{n-1}$, 
arranged in an increasing order
with $\lambda_i\to\infty$ as $i\to\infty$, and let  $\{X_i\}$ be 
a sequence of the corresponding normalized eigenfunctions of 
$-\Delta_\theta$ on $L^2(\mathbb S^{n-1})$, i.e., for each $i\ge 0$, 
\begin{equation}\label{eq-Laplacian-eigenpair}
-\Delta_{\theta}X_i=\lambda_iX_i.\end{equation}
Here, the multiplicity is considered. Hence, 
$$\lambda_0=0, \,\, \lambda_1=\cdots=\lambda_n=n-1, \,\,\lambda_{n+1}=2n, \,\,\cdots.$$
Note that each $X_i$ is a spherical harmonic of certain degree. In the following, 
we fix such a sequence $\{X_i\}$, which forms an orthonormal basis in $L^2(\mathbb S^{n-1})$. 

For each $i\ge 0$ and any $\eta=\eta(t)\in C^2(\mathbb R)$, we write 
\begin{equation}\label{eq-U2-01a}L(\eta X_i)=(L_i\eta)X_i.\end{equation} 
By \eqref{eq-Laplacian-eigenpair}, we have 
\begin{equation}\label{eq-U2-01b}
L_i\eta=\eta''+\Big(\frac14n(n+2)\psi^{\frac{4}{n-2}}-\frac14(n-2)^2-\lambda_i\Big)\eta.\end{equation}

We have the following results concerning the kernel of $L_i$ for each $i=0, 1, \cdots$.

\begin{lemma}\label{lemma-Asymptotics-U1a}
Let $\psi$ be the positive constant solution of \eqref{eq-psi}. 

$\mathrm{(i)}$  For $i=0$, $\mathrm{Ker}(L_0)$ has a basis $\cos(\sqrt{n-2}t)$ and $\sin(\sqrt{n-2}t)$. 

$\mathrm{(ii)}$ There exists an increasing sequence of positive constants $\{\rho_i\}_{i\ge 1}$, 
divergent to $\infty$, such that  
for any $i\ge 1$, 
$\mathrm{Ker}(L_i)$ has a basis  
$e^{-\rho_i t}$ and $e^{\rho_i t}$. Moreover, $\rho_1=\cdots=\rho_n=1$. 
\end{lemma}

\begin{lemma}\label{lemma-Asymptotics-U2a}
Let $\psi$ be a positive nonconstant periodic solution of \eqref{eq-psi}. 

$\mathrm{(i)}$  For $i=0$, $\mathrm{Ker}(L_0)$ has a basis  
$p_0^+$ and $atp_0^++p_0^-$, for some smooth periodic functions $p_0^+$ and $p_0^-$ on $\mathbb R$,  
and some constant $a$. 

$\mathrm{(ii)}$ There exists an increasing sequence of positive constants $\{\rho_i\}_{i\ge 1}$, 
divergent to $\infty$, such that  
for any $i\ge 1$, 
$\mathrm{Ker}(L_i)$ has a basis 
$e^{-\rho_i t}p_i^+$ and $e^{\rho_i t}p_i^-$, for some smooth periodic functions 
$p_i^+$ and $p_i^-$  on $\mathbb R$. Moreover, $\rho_1=\cdots=\rho_n=1$. 

In addition, all periodic functions in $\mathrm{(i)}$ and $\mathrm{(ii)}$ have the same period as $\psi$.
\end{lemma}

Refer to \cite{KorevaarMPS1999}, \cite{MazzeoP1999},  and \cite{MazzeoDU1996} for details, or 
to \cite{HLL201?} and \cite{HanLi2010} for a more general setting. 
The sequence $\{\rho_i\}$ is commonly referred to as the collection of {\it indicial roots} of the linearized equation.

We denote by $\mathbb Z_+$ the collection of nonnegative integers. 
Define the {\it index set} $\mathcal I$ associated with $\psi$ by  
\begin{align}\label{eq-def-index}
\mathcal I=\Big\{\sum_{i\ge 1} m_i\rho_i;\, m_i\in \mathbb Z_+\text{ with finitely many }m_i>0\Big\}.
\end{align}
In other words, $\mathcal I$ is the collection of linear combinations of finitely many $\rho_1, \rho_2, \cdots$ 
with positive integer coefficients. 
It is possible that some $\rho_i$ can be written 
as a linear combination of some of $\rho_1, \cdots, \rho_{i-1}$ with positive integer coefficients, 
whose sum is at least two. If we write $\mathcal I=\{\mu_i\}$ as an increasing sequence, 
$\{\mu_i\}$ is exactly the sequence in \eqref{eq-estimate-u-k}. 

\smallskip

We next introduce weighted H\"older spaces in $[t_0,\infty)\times\mathbb{S}^{n-1}$. Fix a $t_0>0$.  
For each nonnegative integer $k$, $\alpha\in (0,1)$, and $\mu\in\mathbb{R}$, set
\begin{equation*}
\|w\|_{{C}^{k}_{\mu}([t_0,\infty)\times\mathbb{S}^{n-1})}
=\sum_{j=0}^{k}\sup_{(t,\theta)\in [t_0,\infty)\times\mathbb{S}^{n-1}}e^{\mu t}|\nabla^{j}w(t,\theta)|,
\end{equation*}
and
\begin{align*}
\|w\|_{{C}^{k,\alpha}_{\mu}([t_0,\infty)\times\mathbb{S}^{n-1})}=  
\|w\|_{{C}^{k}_{\mu}([t_0,\infty)\times\mathbb{S}^{n-1})}
+\sup_{t\geq t_0+1}e^{\mu t} [\nabla^kw]_{C^\alpha([t-1,t+1]\times\mathbb{S}^{n-1})},
\end{align*}
where $[\,\cdot\,]_{C^\alpha}$ is the usual H\"older semi-norm.

\begin{definition} 
We define the weighted H\"older space 
${C}^{k,\alpha}_{\mu}([t_0,\infty)\times\mathbb{S}^{n-1})$
 as the collection of functions $w$ in ${C}^{k}([t_0,\infty)\times\mathbb{S}^{n-1})$
 with a finite $\|w\|_{{C}^{k,\alpha}_{\mu}([t_0,\infty)\times\mathbb{S}^{n-1})}$.
\end{definition}

We can define corresponding norms and weighted H\"older spaces for functions depending on $t$ only. 

Let $\psi$ be a positive periodic solution of \eqref{eq-psi}, 
$L$ be the linear operator given as in \eqref{eq-linearization}, 
and $\mu>0$. 
For some $f\in {C}^{0,\alpha}_{\mu}([t_0,\infty)\times\mathbb{S}^{n-1})$, we consider 
the linear equation
\begin{equation}\label{linyamabe}
Lw=f \quad \text{in} \quad (t_0,\infty)\times\mathbb{S}^{n-1}.
\end{equation}
We will  introduce a suitable boundary condition on $t=t_0$ such that
\begin{equation*}
L: {C}^{2,\alpha}_{\mu}([t_0,\infty)\times\mathbb{S}^{n-1})\to {C}^{0,\alpha}_{\mu}([t_0,\infty)\times\mathbb{S}^{n-1})
\end{equation*}
has a bounded inverse.

We first study  the Dirichlet boundary-value problem
\begin{align}\label{linyamabeboundary}\begin{split}
   Lw&=f \quad\text{in } (t_0,\infty)\times\mathbb{S}^{n-1},  \\
   w&=\varphi \quad\text{on }\{t_0\}\times\mathbb{S}^{n-1}.
\end{split}
\end{align}
We point out that the maximum principle cannot be applied directly to the operator $L$
due to the \lq\lq incorrect\rq\rq\, sign of the coefficient of the zero order term. 
Nevertheless, we still have the uniqueness of exponentially decaying solutions. 
Recall that 
$\{X_i\}_{i\ge 0}$ is a fixed sequence of spherical harmonics on $\mathbb S^{n-1}$, 
which forms an orthonormal basis in $L^2(\mathbb S^{n-1})$, and that 
$\{\rho_i\}_{i\ge 1}$ is the sequence given in 
Lemma \ref{lemma-Asymptotics-U1a} and Lemma \ref{lemma-Asymptotics-U2a}.
We define $\rho_0=0$.

\begin{lemma}\label{lemma-uniqueness} Let $\mu>1$, 
$f\in {C}^{0}_{\mu}([t_0,\infty)\times\mathbb{S}^{n-1})$, and 
$\varphi\in C^{0}(\mathbb{S}^{n-1})$. Then, there is at most one solution 
$w\in {C}^{2}_{\mu}([t_0,\infty)\times\mathbb{S}^{n-1})$  of \eqref{linyamabeboundary}. 
\end{lemma}

\begin{proof} We consider $f=0$ and $\varphi=0$, and assume 
$w\in  {C}^{2}_{\mu}([t_0,\infty)\times\mathbb{S}^{n-1})$ is a solution of
the corresponding \eqref{linyamabeboundary}. 
Set, for each $i\ge 0$, 
\begin{equation*}
w_{i}(t)=\int_{\mathbb S^{n-1}}w(t,\theta)X_{i}(\theta)d\theta.
\end{equation*}
Then, $L_i w_{i}=0$ on $(t_0,\infty)$ and $w_{i}(t_0)=0$.
Hence, ${w}_i$ is a linear combination of the basis of Ker$(L_i)$ 
as in Lemma \ref{lemma-Asymptotics-U1a} and Lemma \ref{lemma-Asymptotics-U2a}. 
In particular, ${w}_0$ has at most a linear growth and, for $i\ge 1$, 
$${w}_i(t)=c_1(t)e^{-\rho_it}+c_2(t)e^{\rho_it},$$
where $c_1$ and $c_2$ are periodic functions. 
By the assumption, we have, for any $t\in(t_0,\infty)$,  
\begin{equation*}
|e^{\mu t}{w}_i(t)|\le C.
\end{equation*}
Hence, ${w}_{i}=0$ for any $i$ with $\rho_i<\mu$. We now take any $i$ with $\rho_i\ge\mu$. 
Then, $c_2=0$ and hence ${w}_i(t)=c_1(t)e^{-\rho_it}$, which decays exponentially 
as $t\to \infty$. 
With $ w_i(t_0)=0$, we have
\begin{equation*}
\int_{t_0}^{\infty}\Big[(\partial_t{w}_{i})^2+\Big(\lambda_i
+\frac{(n-2)^2}{4}-\frac{n(n+2)}{4}{\psi}^{\frac{4}{n-2}}\Big){w}_{i}^2\Big]dt=0.
\end{equation*}
Since $\rho_i\ge\mu>1$, then
$\lambda_i\ge 2n$ for such $i$. With $0<{\psi}<1$, we have ${w}_i=0$. 
In conclusion, ${w}_i=0$ for any $i$, and hence ${w}=0$.
\end{proof} 

We next estimate the $C^{2,\alpha}$-norm of solutions of \eqref{linyamabeboundary}.

\begin{lemma}\label{lemma-estimate-PDE} Let $\alpha\in (0,1)$, $\mu>0$, 
$f\in {C}^{0,\alpha}_{\mu}([t_0,\infty)\times\mathbb{S}^{n-1})$, and 
$\varphi\in C^{2,\alpha}(\mathbb{S}^{n-1})$. Suppose 
$w\in {C}^{2,\alpha}_{\mu}([t_0,\infty)\times\mathbb{S}^{n-1})$ is a solution of \eqref{linyamabeboundary}. Then, 
\begin{align}\label{eq-estimate-a-priori}\begin{split}
\|w\|_{{C}^{2,\alpha}_{\mu}([t_0,\infty)\times\mathbb{S}^{n-1})}
&\leq     C \{\|w\|_{{C}^{0}_{\mu}([t_0,\infty)\times\mathbb{S}^{n-1})}
+\|f\|_{{C}^{0,\alpha}_{\mu}([t_0,\infty)\times\mathbb{S}^{n-1})}   \\
&\qquad\quad+e^{\mu t_0}\|\varphi\|_{{C}^{2,\alpha}(\mathbb{S}^{n-1})}\},
\end{split}\end{align}
where $C$ is a positive constant depending only on $n$, $\alpha$, 
$\mu$, and $\psi$, independent of $t_0$.\end{lemma}

\begin{proof} The proof is based on a combination of interior Schauder estimates and boundary 
Schauder estimates. Fix a $t\ge t_0$. We consider two cases. 

First, consider $t>t_0+2$. By the interior Schauder estimate, we have
\begin{align*}
&\,   \sum_{j=0}^{2}\sup_{\mathbb{S}^{n-1}}|\nabla^{j}w(t,\cdot)|+[\nabla^2w]_{C^\alpha([t-1,t+1]\times\mathbb{S}^{n-1})}\\
\leq    &\,    C\{\|w\|_{{L}^{\infty}([t-2,t+2]\times\mathbb{S}^{n-1})}+\|f\|_{{L}^{\infty}([t-2,t+2]\times\mathbb{S}^{n-1})}    
+[f]_{C^\alpha([t-2,t+2]\times\mathbb{S}^{n-1})}\},
\end{align*}
where $C$ is a positive constant independent of $t$.
To estimate the H\"older semi-norm of $f$ on $[t-2,t+2]\times\mathbb{S}^{n-1}$ in the right-hand side, we take 
$(t_1,\theta_1),(t_2,\theta_2)\in[t-2,t+2]\times\mathbb{S}^{n-1}$ with $(t_1,\theta_1)\neq(t_2,\theta_2)$. 
We consider two cases: $|t_1-t_2|\le 2$ and $|t_1-t_2|>2$. In the first case, there exists 
a $t'\in[t-1,t+1]$ such that $t_1, t_2\in [t'-1,t'+1]$. 
Hence, 
$$[f]_{C^\alpha([t-2,t+2]\times\mathbb{S}^{n-1})}
\le \max\{\sup_{t'\in[t-1,t+1]}[f]_{C^\alpha([t'-1,t'+1]\times\mathbb{S}^{n-1})}, 
\|f\|_{{L}^{\infty}([t-2,t+2]\times\mathbb{S}^{n-1})}
\}.$$
Then, 
\begin{align*}
&\,   \sum_{j=0}^{2}\sup_{\mathbb{S}^{n-1}}|\nabla^{j}w(t,\cdot)|+[\nabla^2w]_{C^\alpha([t-1,t+1]\times\mathbb{S}^{n-1})}\\
\leq    &\,    C\{\|w\|_{{L}^{\infty}([t-2,t+2]\times\mathbb{S}^{n-1})}+\|f\|_{{L}^{\infty}([t-2,t+2]\times\mathbb{S}^{n-1})}    
+\sup_{t'\in[t-1,t+1]}[f]_{C^\alpha([t'-1,t'+1]\times\mathbb{S}^{n-1})}\}.
\end{align*}
Multiplying both sides by $e^{\mu t}$ and taking supremum over $t\in (t_0+2, \infty)$, we have
\begin{equation}\label{estimate1}
\begin{split}
&\,   \sum_{j=0}^{2}\sup_{t\in (t_0+2, \infty)}\sup_{\mathbb{S}^{n-1}}e^{\mu t}|\nabla^{j}w(t,\cdot)|
+\sup_{t\in (t_0+2, \infty)}e^{\mu t}[\nabla^2w]_{C^\alpha([t-1,t+1]\times\mathbb{S}^{n-1})}\\
\leq    &\,  C\{ \|w\|_{{C}^{0}_{\mu}([t_0,\infty)\times\mathbb{S}^{n-1})}+ \|f\|_{{C}^{0,\alpha}_{\mu}([t_0,\infty)\times\mathbb{S}^{n-1})}\},
\end{split}
\end{equation}
where $C$ is a positive constant independent of $t_0$.

Next, consider $t_0\leq t\leq t_0+2$. By the boundary Schauder estimate, we have
\begin{align*}
&\,   \sum_{j=0}^{2}\sup_{\mathbb{S}^{n-1}}|\nabla^{j}w(t,\cdot)|+[\nabla^2w]_{C^\alpha([t_0,t_0+3]\times\mathbb{S}^{n-1})}\\
\leq    &\,    C\{ \|w\|_{{L}^{\infty}([t_0,t_0+4]\times\mathbb{S}^{n-1})}
+\|\varphi\|_{C^{2,\alpha}(\mathbb{S}^{n-1})}\\
&\qquad +\|f\|_{{L}^{\infty}([t_0,t_0+4]\times\mathbb{S}^{n-1})}  
+[f]_{C^\alpha([t_0,t_0+4]\times\mathbb{S}^{n-1})}\}.
\end{align*}
By  arguing similarly as above, we have 
\begin{equation}\label{estimate2}
\begin{split}
&\,   \sum_{j=0}^{2}\sup_{t\in [t_0,t_0+2]}\sup_{\mathbb{S}^{n-1}}e^{\mu t}|\nabla^{j}w(t,\cdot)|
+\sup_{t\in[t_0+1,t_0+2]}e^{\mu t}[\nabla^2w]_{C^\alpha([t-1,t+1]\times\mathbb{S}^{n-1})}\\
\leq &\,   C \{\|w\|_{{C}^{0}_{\mu}([t_0,\infty)\times\mathbb{S}^{n-1})}+\|f\|_{{C}^{0,\alpha}_{\mu}([t_0,\infty)\times\mathbb{S}^{n-1})} 
+e^{\mu t_0}\|\varphi\|_{{C}^{2,\alpha}(\mathbb{S}^{n-1})}\}.
\end{split}
\end{equation}

Combining \eqref{estimate1} and \eqref{estimate2}, we have the desired result.
\end{proof}

Next, we estimate the $L^\infty$-norm of solutions of \eqref{linyamabeboundary} on 
cylinders with finite length, 
with homogeneous boundary conditions. The proof below, based on a rescaling argument,  
and some other arguments in this paper  
are adapted from \cite{MazzeoP1999}.

\begin{lemma}\label{lemma-estimate-L-infty} 
Let $\mu>1$ and $\mu\neq \rho_i$ for any $i\ge 1$,  
$t_0$ and $T$ be constants with $t_0\ge 0$ 
and $T-t_0\ge 4$, and
$f\in {C}^{0}([t_0,T]\times\mathbb{S}^{n-1})$. Suppose 
$w\in {C}^{2}([t_0,T]\times\mathbb{S}^{n-1})$ satisfies 
$\int_{\mathbb S^{n-1}}{w}(t,\theta)X_{i}(\theta)d\theta=0$ for any $\rho_i<\mu$ and any $t\in [t_0,T]$, and 
\begin{align*}
   Lw&=f \quad\text{in } (t_0,T)\times\mathbb{S}^{n-1},  \\
   w&=0 \quad\text{on }(\{t_0\}\cup\{T\})\times\mathbb{S}^{n-1}.
\end{align*}
Then, 
\begin{equation}\label{eq-estimate-finite-interval-a}
\sup_{(t,\theta)\in[t_0, T]\times\mathbb{S}^{n-1}}e^{\mu t}|{w}(t,\theta)|
\leq C\sup_{(t,\theta)\in[t_0, T]\times\mathbb{S}^{n-1}}e^{\mu t}|f(t,\theta)|,
\end{equation}
where $C$ is a positive constant depending only on $n$, $\alpha$, 
$\mu$, and $\psi$, independent of $t_0$ and $T$. \end{lemma}

\begin{proof} We prove by a contradiction argument. Suppose there exist sequences $\{t_i\}$,  $\{T_i\}$, 
$\{w_i\}$, and $\{f_i\}$, with $t_i\ge 0$ and $T_i-t_i\ge 4$,  such that 
\begin{align*}
   Lw_i&=f_i \quad\text{in } (t_i,T_i)\times\mathbb{S}^{n-1},  \\
   w_i&=0 \quad\text{on }(\{t_i\}\cup\{T_i\})\times\mathbb{S}^{n-1},
\end{align*}
and 
\begin{align*}
\sup_{(t,\theta)\in[t_i, T_i]\times\mathbb{S}^{n-1}}e^{\mu t}|f_i(t,\theta)|&=1,\\
\sup_{(t,\theta)\in[t_i, T_i]\times\mathbb{S}^{n-1}}e^{\mu t}|{w}_i(t,\theta)|&\to\infty\quad\text{as }i\to\infty. 
\end{align*}
Choose $t_i^*\in(t_i,T_i)$ such that
\begin{equation*}
A_i\equiv\sup_{\mathbb{S}^{n-1}}e^{\mu t^*_i}|w_i(t^*_i,\cdot)|=\sup_{(t,\theta)\in[t_0, T]\times\mathbb{S}^{n-1}}e^{\mu t}|{w}_i(t,\theta)|. 
\end{equation*}
Then, $A_i\to\infty$ as $i\to\infty$. Define
\begin{align}\label{eq-definition-rescaling-w}
\widetilde{w}_i(t,\theta)=A_i^{-1}e^{\mu t_i^*}w_i(t+t_i^*,\theta), \end{align}
and 
\begin{align}\label{eq-definition-rescaling-f}
\widetilde{f}_i(t,\theta)&=A_i^{-1}e^{\mu t_i^*}f_i(t+t_i^*,\theta).
\end{align}
Then, 
\begin{equation*}
\sup_{\mathbb S^{n-1}}|\widetilde{w}_i(0,\cdot)|=1,
\end{equation*}
and, for any $(t,\theta)\in [t_i-t_i^*,T_i-t_i^*]\times\mathbb{S}^{n-1}$, 
\begin{equation}\label{eq-estimate-w-tilde}
|e^{\mu t}\widetilde{w}_i(t,\theta)|\le 1.
\end{equation}
Moreover, 
\begin{equation*}
L_{\psi(\cdot+t_i^*)}\widetilde{w}_i=\widetilde{f}_i\quad\text{on }(t_i-t_i^*,T_i-t_i^*)\times\mathbb{S}^{n-1}, 
\end{equation*}
where $L_{\psi(\cdot+t_i^*)}$ is the linearized Yamabe operator at $\psi(\cdot+t_i^*)$. 
Passing to subsequences, we assume, for some $\tau_-\in\mathbb{R}^{-}\cup\{-\infty\}$
and $\tau_+\in\mathbb{R}^{+}\cup\{\infty\}$,  
\begin{equation}\label{eq-limit-tau}t_i-t_i^*\to \tau_-, \quad T_i-t_i^*\to \tau_+.\end{equation}
Hence, $\tau_-<0$ if it is finite, and similarly $\tau_+>0$ if it is finite. In fact, it follows from \eqref{eq-estimate-w-tilde} that 
$$|\widetilde{w}_{i}|
\le Ce^{\mu (t_i^*-t_i)} \quad\text{on }(t_i-t_i^*,t_i-t_i^*+2)\times\mathbb{S}^{n-1},$$ 
and hence
$$\Big|\frac{d^2\widetilde{w}_{i}}{dt^2}+\Delta_{\theta}\widetilde{w}_{i}\Big|
\le Ce^{\mu (t_i^*-t_i)} \quad\text{on }(t_i-t_i^*,t_i-t_i^*+2)\times\mathbb{S}^{n-1}.$$ 
Since $\widetilde{w}_{i}=0$ on $\{t_i-t_i^*\}\times\mathbb{S}^{n-1}$,
we  have 
$$|\nabla\widetilde{w}_{i}|
\le Ce^{\mu (t_i^*-t_i)} \quad\text{on }(t_i-t_i^*,t_i-t_i^*+1)\times\mathbb{S}^{n-1}.$$ 
This proves that $t_i-t_i^*$ remains bounded away from zero. Similarly, 
$T_i-t_i^*$ remains bounded away from zero. As a consequence, $0\in (\tau_-, \tau_+)$. 
Furthermore, we  assume
\begin{equation}\label{eq-limit-w}\widetilde{w}_i\to \widehat{w},\quad 
\psi(\cdot+t_i^*)\to\widehat{\psi}\quad\text{in every compact set of }(\tau_-, \tau_+).\end{equation}
We also have $\widetilde{f}_i\to 0$  in every compact set of $(\tau_-, \tau_+)$. Therefore,  $\widehat{w}\neq 0$, 
\begin{equation}\label{eq-estimate-w}
|e^{\mu t}\widehat{w}(t,\theta)|\le 1\quad\text{for any }(t,\theta)\in(\tau_-, \tau_+)\times\mathbb{S}^{n-1},
\end{equation}
and 
\begin{equation*}
L_{\widehat\psi}\widehat{w}=0\quad\text{on }(\tau_-, \tau_+)\times\mathbb{S}^{n-1},
\end{equation*}
where $L_{\widehat\psi}$ is the linearized Yamabe operator at $\widehat\psi$. 
We note that $\widehat\psi$ is a positive periodic solution of \eqref{eq-psi}. In fact, 
$\widehat\psi=\psi(\cdot+\widehat \tau)$ for some $\widehat \tau$. 
Moreover, 
\begin{equation}\label{eq-limit-w-value}\lim_{t\to \tau_*}\widehat w(t)=0,\end{equation}
where $\tau_*=\tau_-$ or $\tau_+$ if it is finite. 

Next, we proceed as in the proof of Lemma \ref{lemma-uniqueness}. For any $i\ge 0$, set 
\begin{equation}\label{eq-definition-w-i}
\widehat{w}_{i}(t)=\int_{\mathbb S^{n-1}}\widehat{w}(t,\theta)X_{i}(\theta)d\theta.
\end{equation}
Then, $L_i\widehat{w}_{i}=0$ and hence, $\widehat{w}_i$ is a linear combination of the basis of Ker$(L_i)$ 
as in Lemma \ref{lemma-Asymptotics-U1a} and Lemma \ref{lemma-Asymptotics-U2a}. 
By the assumption, 
$\widehat{w}_i=0$ for any $i$ with $\rho_i<\mu$. 
We now take an $i$ with $\rho_i>\mu$. Then, 
$$\widehat{w}_i(t)=c_1(t)e^{-\rho_it}+c_2(t)e^{\rho_it},$$
where $c_1$ and $c_2$ are periodic functions. 
By \eqref{eq-estimate-w}, we have, for any $t\in(\tau_-,\tau_+)$,  
\begin{equation*}
|e^{\mu t}\widehat{w}_i(t)|\le C.
\end{equation*}
If $\tau_+=\infty$, 
then $c_2=0$ and hence $\widetilde{w}_i(t)=c_1(t)e^{-\rho_it}$, which decays exponentially 
as $t\to \infty$. 
If $\tau_+$ is finite, then $\lim_{t\to\tau_+}\widehat{w}_{i}(t)=0$
by \eqref{eq-limit-w-value}.
Similarly, if $\tau_-=-\infty$, then $c_1=0$ and hence $\widehat{w}_i(t)=c_2(t)e^{\rho_it}$, 
which decays exponentially 
as $t\to -\infty$. 
If $\tau_-$ is finite, then $\lim_{t\to\tau_-}\widehat{w}_{i}(t)=0$
by \eqref{eq-limit-w-value}.
Thus, 
\begin{equation*}
\int_{\tau_-}^{\tau_+}\Big[(\partial_t\widehat{w}_{i})^2+\Big(\lambda_i
+\frac{(n-2)^2}{4}-\frac{n(n+2)}{4}\widehat{\psi}^{\frac{4}{n-2}}\Big)\widehat{w}_{i}^2\Big]dt=0.
\end{equation*}
Since $\rho_i>\mu>1$, then
$\lambda_i\ge 2n$ for such $i$. With $0<\widehat{\psi}<1$, we have $\widehat{w}_i=0$. 
In conclusion, $\widehat{w}_i=0$ for any $i$ and hence $\widehat{w}=0$, which leads to a contradiction.
\end{proof}

Now we start to construct suitable solutions of \eqref{linyamabe}. 
We first construct exponentially decaying solutions in appropriate finite dimensional subspaces in $L^2(\mathbb S^{n-1})$. 

\begin{lemma}\label{lemma-solution-finite-dim} Let $\alpha\in (0,1)$, $\mu>\rho_I$ for some $I\ge 1$, 
and $f\in {C}^{0,\alpha}_{\mu}([t_0,\infty)\times\mathbb{S}^{n-1})$ with $f(t, \cdot)\in 
\mathrm{span}\{X_0, X_1, \cdots, X_I\}$ for any $t\ge t_0$. Then, there exists a  unique solution 
$w\in  {C}^{2,\alpha}_{\mu}([t_0,\infty)\times\mathbb{S}^{n-1})$ of \eqref{linyamabe} with $w(t, \cdot)\in 
\mathrm{span}\{X_0, X_1, \cdots, X_I\}$ for any $t\ge t_0$. Moreover, the correspondence $f\mapsto w$ is linear, 
and 
\begin{equation*}
\|w\|_{{C}^{2,\alpha}_{\mu}([t_0,\infty)\times\mathbb{S}^{n-1})}
\leq C\|f\|_{{C}^{0,\alpha}_{\mu}([t_0,\infty)\times\mathbb{S}^{n-1})},
\end{equation*}
where $C$ is a positive constant depending only on $n$, $\alpha$, 
$\mu$, and $\psi$, independent of $t_0$.
\end{lemma}

\begin{proof} 
For each $i=0, 1, \cdots, I$, we set 
$$f_{i}(t)=\int_{\mathbb{S}^{n-1}}f(t,\theta)X_{i}(\theta)d\theta.$$
Then, 
$$\|f_i\|_{{C}^{0,\alpha}_{\mu}([t_0,\infty))}\le C\|f\|_{{C}^{0,\alpha}_{\mu}([t_0,\infty)\times\mathbb{S}^{n-1})},$$
and 
\begin{equation}\label{eq-series-f}
f(t,\theta)=\sum_{i=0}^{I}f_{i}(t)X_{i}(\theta).
\end{equation}
Let $L_i$ be the linear operator given as in \eqref{eq-U2-01b}. 

We first consider the ordinary differential equation 
\begin{equation}\label{li}
L_iw_{i}=f_{i}.
\end{equation}
We claim that there exists a solution 
$w_i\in {C}^{2,\alpha}_{\mu}([t_0,\infty))$ of \eqref{li} satisfying 
\begin{equation}\label{eq-C2alpha-ODE}
\|w_i\|_{{C}^{2,\alpha}_{\mu}([t_0,\infty))}
\leq C\|f_i\|_{{C}^{0,\alpha}_{\mu}([t_0,\infty))},
\end{equation}
where $C$ is a constant depending only on $n$, $\alpha$, $\mu$, and $\psi$, independent of $t_0$. 
By the classification of global radial solutions, $\psi$ is either a positive constant (which is unique) or 
a positive nonconstant periodic function. We consider the latter case. The former case is easier. 

Consider first $i>0$. By Lemma \ref{lemma-Asymptotics-U2a}(ii), the kernel $\mathrm{Ker}(L_i)$ is spanned by 
$\psi_{i}^{+}(t)=e^{-\rho_i t}p_i^+(t)$ and  $\psi_{i}^{-}(t)=e^{\rho_i t}p_i^-(t)$,
for some periodic functions $p_i^+$ and  $p_i^-$. 
Set
\begin{equation}\label{eq-expression-w-i}
{w}_{i}(t)=\psi_{i}^{+}(t)\int_{t}^{\infty}\frac{\psi_{i}^{-}(s)}{W(s)}f_{i}(s)ds
-\psi_{i}^{-}(t)\int_{t}^{\infty}\frac{\psi_{i}^{+}(s)}{W(s)}f_{i}(s)ds,
\end{equation}
where $W$ is the Wronskian determinant given by
\begin{equation*}
W=\psi^{+}_{i}(\psi^{-}_{i})'-\psi^{-}_{i}(\psi^{+}_{i})'.
\end{equation*}
A simple computation yields, for $t\geq t_0$,
\begin{equation}\label{eq-esti-0-w}
e^{\mu t}|{w}_{i}(t)|
\leq    C \sup_{t\geq t_0} e^{\mu t}|f_{i}(t)|=C \|f_i\|_{{C}^{0}_{\mu}([t_0,\infty))}.
\end{equation}
By a straightforward computation, we have
\begin{equation*}
{w}'_{i}(t)=(\psi_{i}^{+})'(t)\int_{t}^{\infty}\frac{\psi_{i}^{-}(s)}{W(s)}f_{i}(s)ds
-(\psi_{i}^{-})'(t)\int_{t}^{\infty}\frac{\psi_{i}^{+}(s)}{W(s)}f_{i}(s)ds,
\end{equation*}
and
\begin{align*}
{w}''_{i}(t)&=      (\psi_{i}^{+})''(t)\int_{t}^{\infty}\frac{\psi_{i}^{-}(s)}{W(s)}f_{i}(s)ds
-(\psi_{i}^{-})''(t)\int_{t}^{\infty}\frac{\psi_{i}^{+}(s)}{W(s)}f_{i}(s)ds   \\
&\qquad-(\psi_{i}^{+})'(t)\frac{\psi_{i}^{-}(t)}{W(t)}f_{i}(t)+(\psi_{i}^{-})'(t)\frac{\psi_{i}^{+}(t)}{W(t)}f_{i}(t).
\end{align*}
Since $\psi_{i}^{+}$, $\psi_{i}^{-}$ are multiples of periodic functions by $e^{-\rho_i t}$, $e^{\rho_i t}$, 
so are their derivatives. Similarly, we obtain, for $t\geq t_0$,
\begin{equation}\label{eq-esti-1-2-w}
e^{\mu t}|{w}'_{i}(t)|+e^{\mu t}|{w}''_{i}(t)|\leq C \|f_i\|_{{C}^{0}_{\mu}([t_0,\infty))}.
\end{equation}
For the H\"older semi-norms 
of $w_i''$, we write 
$${w}''_{i}=P_1+P_2,$$ 
where 
\begin{equation*}
P_1(t)=(\psi_{i}^{+})''(t)\int_{t}^{\infty}\frac{\psi_{i}^{-}(s)}{W(s)}f_{i}(s)ds-(\psi_{i}^{-})''(t)\int_{t}^{\infty}\frac{\psi_{i}^{+}(s)}{W(s)}f_{i}(s)ds,
\end{equation*}
and
\begin{equation*}
P_2(t)=-(\psi_{i}^{+})'(t)\frac{\psi_{i}^{-}(t)}{W(t)}f_{i}(t)+(\psi_{i}^{-})'(t)\frac{\psi_{i}^{+}(t)}{W(t)}f_{i}(t).
\end{equation*}
Then,
\begin{align*}
P'_1(t)&=      (\psi_{i}^{+})'''(t)\int_{t}^{\infty}\frac{\psi_{i}^{-}(s)}{W(s)}f_{i}(s)ds
-(\psi_{i}^{-})'''(t)\int_{t}^{\infty}\frac{\psi_{i}^{+}(s)}{W(s)}f_{i}(s)ds   \\
&\qquad-(\psi_{i}^{+})''(t)\frac{\psi_{i}^{-}(t)}{W(t)}f_{i}(t)+(\psi_{i}^{-})''(t)\frac{\psi_{i}^{+}(t)}{W(t)}f_{i}(t).
\end{align*}
Similarly, we have, for $t\geq t_0$,
\begin{equation*}
e^{\mu t}|P'_1(t)|\leq C \|f_i\|_{{C}^{0}_{\mu}([t_0,\infty))},
\end{equation*}
and hence, for $t\geq t_0+1$,
\begin{align*}
e^{\mu t}[P_1]_{C^\alpha([t-1,t+1])}   
\leq      C \|f_i\|_{{C}^{0}_{\mu}([t_0,\infty)\times\mathbb{S}^{n-1})}.
\end{align*}
Since $(\psi^{+}_{i})'\psi^{-}_{i}/{W}$ and $(\psi^{-}_{i})'\psi^{+}_{i}/{W}$ are periodic, we have, for $t\geq t_0+1$,
\begin{align*}
e^{\mu t}[P_2]_{C^\alpha([t-1,t+1])}   
\leq      C \|f_i\|_{{C}^{0, \alpha}_{\mu}([t_0,\infty)\times\mathbb{S}^{n-1})}.
\end{align*}
Therefore, for $t\geq t_0+1$,
\begin{align}\label{eq-Holder-semi}
e^{\mu t}[w_i'']_{C^\alpha([t-1,t+1])}   
\leq      C \|f_i\|_{{C}^{0, \alpha}_{\mu}([t_0,\infty)\times\mathbb{S}^{n-1})}.
\end{align}
By combining \eqref{eq-esti-0-w}, \eqref{eq-esti-1-2-w}, and \eqref{eq-Holder-semi}, we have \eqref{eq-C2alpha-ODE} for $i>0$.

Next, consider  $i=0$. By Lemma \ref{lemma-Asymptotics-U2a}(i), the kernel $\mathrm{Ker}(L_0)$ is spanned by 
$\psi_{0}^{+}(t)=p_0^+(t)$ and  $\psi_{0}^{-}(t)=atp_0^+(t)+p_0^-(t)$,
for some periodic functions $p_{0}^{+}$ and  $p_0^-$, and some constant $a$.  We set
\begin{align}\label{eq-expression-w-0}\begin{split}
{w}_{0}(t)&=    p^{+}_{0}(t)\int_{t}^{\infty}\frac{p_0^-(s)}{W(s)}f_{0}(s)ds-p_0^-(t)\int_{t}^{\infty}\frac{p_{0}^{+}(s)}{W(s)}f_{0}(s)ds  \\
&\qquad  -ap^{+}_{0}(t)\int_{t}^{\infty}\int_{s}^{\infty}\frac{p_{0}^{+}(\tau)}{W(\tau)}f_{0}(\tau)d\tau ds.
\end{split}\end{align}
We can prove \eqref{eq-C2alpha-ODE} for $i=0$ similarly. 

With the solution $w_i$  of \eqref{li} for $i=0, 1, \cdots, I$, we set  
\begin{equation*}
w(t,\theta)=\sum_{i=0}^{I}w_{i}(t)X_{i}(\theta).
\end{equation*}
Then, $Lw=f$ and, by \eqref{eq-C2alpha-ODE}, 
\begin{align*}
\|{w}\|_{{C}^{2, \alpha}_{\mu}([t_0,\infty)\times\mathbb{S}^{n-1})}
&\le \sum_{i=0}^IC\|{w}_{i}\|_{{C}^{2, \alpha}_{\mu}([t_0,\infty))}\\
&\le \sum_{i=0}^IC\|{f}_{i}\|_{{C}^{0, \alpha}_{\mu}([t_0,\infty))}
\leq C\|f\|_{{C}^{0,\alpha}_{\mu}([t_0,\infty)\times\mathbb{S}^{n-1})}.
\end{align*}
Therefore, $w$ is the desired solution. It is easy to see that such a $w$ is unique.
\end{proof}

The expressions \eqref{eq-expression-w-i} and \eqref{eq-expression-w-0} are from \cite{HanLi2010}. 
We point out that the uniqueness of solutions 
$w$ holds only with the extra requirement  $w(t, \cdot)\in 
\mathrm{span}\{X_0, X_1, \cdots, X_I\}$ for any $t\ge t_0$.

We next construct solutions in infinite dimensional subspaces in $L^2(\mathbb S^{n-1})$. 

\begin{lemma}\label{lemma-solution-infinite-dim}
Let $\alpha\in (0,1)$, $\mu>1$  and $\mu\neq \rho_i$ for any $i\ge 1$, 
and $f\in {C}^{0,\alpha}_{\mu}([t_0,\infty)\times\mathbb{S}^{n-1})$, with 
$\int_{\mathbb S^{n-1}}f(t, \cdot)X_id\theta=0$, for any $i$ with $\rho_i<\mu$ and 
any $t\ge t_0$. Then, there exists a unique solution $w\in {C}^{2,\alpha}_{\mu}([t_0,\infty)\times\mathbb{S}^{n-1})$
of \eqref{linyamabe} with $w=0$ on $\{t_0\}\times\mathbb{S}^{n-1}$. Moreover, 
\begin{equation}\label{eq-estimate-C2alpha}
\|w\|_{{C}^{2,\alpha}_{\mu}([t_0,\infty)\times\mathbb{S}^{n-1})}
\leq C\|f\|_{{C}^{0,\alpha}_{\mu}([t_0,\infty)\times\mathbb{S}^{n-1})},
\end{equation}
where $C$ is a positive constant depending only on $n$, $\alpha$, 
$\mu$, and $\psi$, independent of $t_0$.
\end{lemma}

\begin{proof} Take any $T\ge t_0+4$. 
We first prove that there exists a solution ${w}_T\in {C}^{2,\alpha}([t_0,T]\times\mathbb{S}^{n-1})$ of
\begin{align}\label{linyamabelargeT}
\begin{split}
L{w}_T&={f} \quad\text{in } (t_0,T)\times\mathbb{S}^{n-1},\\
{w}_T&=0\quad\text{on } (\{t_0\}\cup\{T\})\times\mathbb{S}^{n-1}.
\end{split}
\end{align}
Consider the energy function
\begin{equation*}
\mathcal{E}_T(w)=\int_{t_0}^{T}\int_{\mathbb{S}^{n-1}}\Big[(\partial_tw)^2+|\nabla_{\theta}w|^2
+\Big(\frac{(n-2)^2}{4}
-\frac{n(n+2)}{4}\psi^{\frac{4}{n-2}}\Big)w^2+2{f}w\Big]dt d\theta.
\end{equation*}
Set 
$$\mathcal X=\Big\{u\in H^1(\mathbb{S}^{n-1});\, \int_{\mathbb S^{n-1}}u(\theta)X_i(\theta) d\theta=0
\text{ for any $i$ with $\rho_i<\mu$}\Big\}.$$
Then, for any $u\in \mathcal X$, 
$$\int_{\mathbb{S}^{n-1}}|\nabla_{\theta}u|^2d\theta\ge 2n 
\int_{\mathbb{S}^{n-1}}u^2d\theta,$$
since $\lambda_k>\lambda_n\geq 2n$ for any $k$ with $\rho_k>\mu>\rho_n=1$. 
Hence, for any $w\in H^1_0((t_0,T)\times\mathbb{S}^{n-1})$ with $w(t,\cdot)\in \mathcal X$ for any 
$t\in (t_0,T)$, we have 
\begin{equation*}
\mathcal{E}_T(w)\geq\int_{t_0}^T\int_{\mathbb{S}^{n-1}}\Big[(\partial_tw)^2
+\Big(\frac{(n+2)^2}{4}-\frac{n(n+2)}{4}\psi^{\frac{4}{n-2}}\Big)w^2
+2{f}w\Big]dtd\theta. 
\end{equation*}
By $0<\psi<1$, we conclude that $\mathcal{E}_T$ is coercive and weak lower semi-continuous. 
Hence, we can find a minimizer ${w}_T$ of $\mathcal E_T$ in the space 
\begin{align*}\{w\in H^1_0((t_0,T)\times\mathbb{S}^{n-1});\, w(t,\cdot)\in \mathcal X
\text{ for any $t\in (t_0,T)$}\}.\end{align*}
Since $f(t,\cdot)\in \mathcal X$ for any 
$t\in (t_0,T)$, $w_T$ is a solution of \eqref{linyamabelargeT}, 
with $w_T(t,\cdot)\in \mathcal X$
for any $t\in (t_0,T)$.

By Lemma \ref{lemma-estimate-L-infty}, we have 
\begin{equation*}
\sup_{(t,\theta)\in[t_0, T]\times\mathbb{S}^{n-1}}e^{\mu t}|{w}_T(t,\theta)|
\leq C\sup_{(t,\theta)\in[t_0, T]\times\mathbb{S}^{n-1}}e^{\mu t}|f(t,\theta)|,
\end{equation*}
where $C$ is a positive constant depending only on $n$, $\alpha$, 
$\mu$, and $\psi$, independent of $t_0$ and $T$. 
For each fixed $T_0>t_0$, 
consider $[t_0,t_0+T_0]\times\mathbb{S}^{n-1}\subset[t_0,t_0+T_0+1]\times\mathbb{S}^{n-1}$. 
By the interior and boundary Schauder estimates, ${w}_T(t_0,\theta)=0$,  and passing to a subsequence, 
${w}_T$ converges to a ${C}^{2,\alpha}$ solution $w$ of \eqref{linyamabe}  in $[t_0,t_0+T_0]\times\mathbb{S}^{n-1}$
with $w=0$ on $\{t_0\}\times\mathbb{S}^{n-1}$, as $T\to\infty$.
By a diagonalization process, ${w}_T$ converges to 
a ${C}^{2,\alpha}$  solution ${w}$ of \eqref{linyamabe} in $[t_0,\infty)\times\mathbb{S}^{n-1}$, 
with $w=0$ on $\{t_0\}\times\mathbb{S}^{n-1}$. Moreover,
\begin{equation*}
\sup_{(t,\theta)\in[t_0, \infty)\times\mathbb{S}^{n-1}}e^{\mu t}|{w}(t,\theta)|
\leq C\sup_{(t,\theta)\in[t_0, \infty)\times\mathbb{S}^{n-1}}e^{\mu t}|f(t,\theta)|,
\end{equation*}
or 
\begin{equation}\label{eq-estimate-infinite-interval}
\|w\|_{{C}^{0}_{\mu}([t_0,\infty)\times\mathbb{S}^{n-1})}\leq C\|f\|_{{C}^{0,\alpha}_{\mu}([t_0,\infty)\times\mathbb{S}^{n-1})},
\end{equation}
where $C$ is a positive constant depending only on $n$, $\alpha$, 
$\mu$, and $\psi$, independent of $t_0$. 
By substituting \eqref{eq-estimate-infinite-interval} in \eqref{eq-estimate-a-priori} with $\varphi=0$, we have 
\eqref{eq-estimate-C2alpha}. 
\end{proof}

Now we are ready to prove the main solvability result in this section. 

\begin{theorem}\label{linthm}
Let $\alpha\in (0,1)$, $\mu>1$ with $\mu\neq \rho_i$ for any $i$, 
and $f\in {C}^{0,\alpha}_{\mu}([t_0,\infty)\times\mathbb{S}^{n-1})$. 
Then, \eqref{linyamabe} admits a solution $w\in {C}^{2,\alpha}_{\mu}([t_0,\infty)\times\mathbb{S}^{n-1})$
and 
\begin{equation}\label{eq-estimate-main-PDE}
\|w\|_{{C}^{2,\alpha}_{\mu}([t_0,\infty)\times\mathbb{S}^{n-1})}
\leq C\|f\|_{{C}^{0,\alpha}_{\mu}([t_0,\infty)\times\mathbb{S}^{n-1})},
\end{equation}
where $C$ is a positive constant depending only on $n$, $\alpha$, 
$\mu$, and $\psi$, independent of $t_0$.
Moreover, the correspondence $f\mapsto w$ is linear. \end{theorem}

\begin{proof} Take $I$ to be the largest integer such that $\rho_I<\mu$. Set, for $i=0, 1, \cdots, I$, 
$$f_i(t)=\int_{\mathbb S^{n-1}}f(t, \theta)X_i(\theta)d\theta.$$ 
First, let $w_1\in {C}^{2,\alpha}_{\mu}([t_0,\infty)\times\mathbb{S}^{n-1})$ be the unique solution of 
$$Lw_1=\sum_{i=0}^If_iX_i\quad\text{in }[t_0,\infty)\times\mathbb{S}^{n-1},$$ 
as in Lemma \ref{lemma-solution-finite-dim}. 
By Lemma \ref{lemma-solution-infinite-dim}, let $w_2\in {C}^{2,\alpha}_{\mu}([t_0,\infty)\times\mathbb{S}^{n-1})$ 
be the unique solution of 
\begin{align*}Lw&=f-\sum_{i=0}^If_iX_i\quad\text{in }[t_0,\infty)\times\mathbb{S}^{n-1},\\
w&=0\quad\text{on }\{t_0\}\times\mathbb{S}^{n-1}.\end{align*} 
Then, $w=w_1+w_2$ is the desired solution. \end{proof}

\begin{remark}\label{remark-inverse} We denote by $L^{-1}$ the 
correspondence $f\mapsto w$ as in Theorem \ref{linthm}. 
Then, 
\begin{equation}\label{linyamabeinverse}
L^{-1}: {C}^{0,\alpha}_{\mu}([t_0,\infty)\times\mathbb{S}^{n-1})\to {C}^{2,\alpha}_{\mu}([t_0,\infty)\times\mathbb{S}^{n-1})
\end{equation}
is a bounded linear operator, and the bound of $L^{-1}$ does not depend on $t_0$.
We emphasize that $L^{-1}$ has a built-in boundary condition on $t=t_0$. 
\end{remark}

\section{Approximate Solutions}\label{sec-approximate-solution}

In this section, we introduce the notion of approximate solutions  
and prove Theorem \ref{mainthm} by the contraction mapping principle. 
We also demonstrate that we can always construct approximate solutions by perturbing 
solutions of the linearized Yamabe equation. 

For convenience, we set 
\begin{equation}\label{eq-yamabe-operator-v}
\mathcal N(v)=v_{tt}+\Delta_{\theta}v-\frac14(n-2)^2v+\frac14n(n-2)v^{\frac{n+2}{n-2}}.\end{equation}
Obviously, $v$ is a solution of \eqref{eq-yamabe-v} if $\mathcal N(v)=0$. 

We now prove the main result in this section. 

\begin{theorem}\label{mainthm-v}
Let $\psi$ be a positive periodic solution of \eqref{eq-psi}, $\mathcal I$ the index set associated with $\psi$, 
and $\mu>1$ with $\mu\notin \mathcal I$. Suppose that $\widehat v\in {C}^{2,\alpha}([0,\infty)\times\mathbb{S}^{n-1})$
satisfies 
\begin{equation}\label{eq-assumption-leading-v}|(\widehat v-\psi)(t, \theta)|+|\nabla(\widehat v-\psi)(t, \theta)|\to 0
\quad\text{as $t\to 0$ uniformly in $\theta\in\mathbb S^{n-1}$},
\end{equation}
and, for any $(t,\theta)\in  [0,\infty)\times \mathbb S^{n-1}$, 
\begin{equation}\label{eq-assumption-approximate-v}
|\mathcal N(\widehat v)(t, \theta)|+|\nabla(\mathcal N(\widehat v))(t, \theta)|\le Ce^{-\mu t},
\end{equation}
for some positive constant $C$. 
Then, there exist a $t_0>0$ and a solution 
$v\in {C}^{2,\alpha}([t_0,\infty)\times\mathbb{S}^{n-1})$ of \eqref{eq-yamabe-v} such that,  
for any $(t,\theta)\in(t_0,\infty)\times\mathbb{S}^{n-1}$,
\begin{equation*}
|v(t,\theta)-\widehat v(t,\theta)|\leq C e^{-\mu t},
\end{equation*}
where $C$ is a positive constant.
\end{theorem}

\begin{proof} The proof consists of several steps. 

{\it Step 1.} We rewrite the equation \eqref{eq-yamabe-v}. 
Let $\mathcal N$ be the operator introduced in \eqref{eq-yamabe-operator-v}. 
We will find $w\in\mathcal{C}^{2,\alpha}_{\mu}([t_0,\infty)\times\mathbb{S}^{n-1})$ such that
\begin{equation}\label{eq-equation-w}
\mathcal N(\widehat v+w)=0.\end{equation}
We rewrite this equation as 
\begin{align*}
Lw+\mathcal N(\widehat v) 
+P(w)
=0,
\end{align*}
where 
\begin{equation}\label{eq-P}
P(w)=\frac{n(n-2)}{4}\big[(\widehat v+w)^{\frac{n+2}{n-2}}-\widehat v^{\frac{n+2}{n-2}}\big]-\frac{n(n+2)}{4}\psi^{\frac{4}{n-2}}w. 
\end{equation}
With the operator $L^{-1}$ introduced in Remark \ref{remark-inverse}, 
we can rewrite it further as 
\begin{align*}
w=      L^{-1}\big[-\mathcal N(\widehat v) 
-P(w)
\big]. 
\end{align*}
We define the mapping $\mathcal{T}$ by
\begin{align}\label{eq-T}
\mathcal{T}(w)= L^{-1}\big[-\mathcal N(\widehat v) 
-P(w)
\big].  
\end{align}
We will prove that $\mathcal{T}$ is a contraction on some ball in ${C}^{2,\alpha}_{\mu}([t_0,\infty)\times\mathbb{S}^{n-1})$,
for some $t_0$ large. For brevity, we set  
$$\mathcal X_{B, t_0}=\{w\in {C}^{2,\alpha}_{\mu}([t_0,\infty)\times\mathbb{S}^{n-1}); 
\|w\|_{{C}^{2,\alpha}_{\mu}([t_0,\infty)\times\mathbb{S}^{n-1})}\le B\}.$$

{\it Step 2.} We prove that $\mathcal{T}$ maps $\mathcal X_{B, t_0}$ to itself, 
for some fixed $B$ and any $t_0$ sufficiently large; namely, for any $w\in {C}^{2,\alpha}_{\mu}([t_0,\infty)\times\mathbb{S}^{n-1})$ 
with $\|w\|_{{C}^{2,\alpha}_{\mu}(t_0,\infty)\times\mathbb{S}^{n-1})}\leq B$,
we have $\mathcal{T}(w)\in{C}^{2,\alpha}_{\mu}([t_0,\infty)\times\mathbb{S}^{n-1})$ and
$\|\mathcal{T}(w)\|_{{C}^{2,\alpha}_{\mu}([t_0,\infty)\times\mathbb{S}^{n-1})}\leq B$.

First, by \eqref{eq-assumption-approximate-v}, we have 
$$\|\mathcal N(\widehat v)\|_{{C}^{1}_{\mu}([t_0,\infty)\times\mathbb{S}^{n-1})}\le C_1.$$ 
Next, set 
\begin{equation}\label{eq-definition-Q}Q(w)=\frac{n(n+2)}{4}\int_{0}^{1}\big[(\widehat v+sw)^{\frac{4}{n-2}}-\psi^{\frac{4}{n-2}}\big]ds.
\end{equation}
Then, $P(w)=wQ(w)$. 
Take any $w\in{C}^{2,\alpha}_{\mu}([t_0,\infty)\times\mathbb{S}^{n-1})$ with
$\|w\|_{{C}^{2,\alpha}_{\mu}([t_0,\infty)\times\mathbb{S}^{n-1})}\leq B$, for some $B$ to be determined.
Note that 
$$|\widehat v-\psi|+|\nabla(\widehat v-\psi)|\le \epsilon(t),$$
where $\epsilon$ is a decreasing function with $\epsilon(t)\to0$ as $t\to \infty$, and 
$$|w|+|\nabla w|\le Be^{-\mu t}.$$ 
Then, for $t\ge t_0$,
\begin{equation}\label{eq-estimate-Q} 
|Q(w)|+|\nabla Q(w)|\le C_2(\epsilon(t)+Be^{-\mu t}),\end{equation}
and hence, 
\begin{align*}\|P(w)\|_{{C}^{1}_{\mu}([t_0,\infty)\times\mathbb{S}^{n-1})}
&\le C_2(\epsilon(t_0)+Be^{-\mu t_0})\|w\|_{{C}^{1}_{\mu}([t_0,\infty)\times\mathbb{S}^{n-1})}\\
&\le C_2(\epsilon(t_0)+Be^{-\mu t_0})B.\end{align*}
By Theorem \ref{linthm}, we get 
\begin{align*}
\|\mathcal{T}(w)\|_{{C}^{2,\alpha}_{\mu}([t_0,\infty)\times\mathbb{S}^{n-1})}
&\leq C \|\mathcal N(\widehat v) 
+P(w)\|_{{C}^{0,\alpha}_{\mu}([t_0,\infty)\times\mathbb{S}^{n-1})}\\
&\le C\big[C_1+C_2(\epsilon(t_0)+Be^{-\mu t_0})B\big],
\end{align*}
where $C$, $C_1$, and $C_2$ are positive constants independent of $t_0$. 
We first take $B\ge 2CC_1$ and then take $t_0$ large such that $CC_2(\epsilon(t_0)+Be^{-\mu t_0})\le 1/2$. Then, 
\begin{align*}
\|\mathcal{T}(w)\|_{{C}^{2,\alpha}_{\mu}([t_0,\infty)\times\mathbb{S}^{n-1})}\le B.\end{align*}
This is the desired estimate.

{\it Step 3.} We prove that $\mathcal{T}: \mathcal X_{B, t_0}\to\mathcal X_{B, t_0}$ 
is a contraction, i.e., for any  $w_1,w_2\in \mathcal X_{B, t_0}$, 
\begin{equation}\label{eq-contraction}
\|\mathcal{T}(w_1)-\mathcal{T}(w_2)\|_{{C}^{2,\alpha}_{\mu}([t_0,\infty)\times\mathbb{S}^{n-1})}
\leq\lambda\|w_1-w_2\|_{{C}^{2,\alpha}_{\mu}([t_0,\infty)\times\mathbb{S}^{n-1})}, 
\end{equation}
for some constant $\lambda\in(0,1)$. 

We note 
$$\mathcal{T}(w_1)-\mathcal{T}(w_2)=-L^{-1}\big[P(w_1)-P(w_2)\big],$$
and 
\begin{align*}P(w_1)-P(w_2)&=w_1Q(w_1)-w_2Q(w_2)\\
&=(w_1-w_2)Q(w_1)+w_2(Q(w_1)-Q(w_2)).\end{align*}
By \eqref{eq-definition-Q}, we have
$$Q(w_1)-Q(w_2)=\frac{n(n+2)}{4}\int_{0}^{1}\big[(\widehat v+sw_1)^{\frac{4}{n-2}}-(\widehat v+sw_2)^{\frac{4}{n-2}}\big]ds.
$$
Then,  
$$|Q(w_1)-Q(w_2)|+|\nabla(Q(w_1)-Q(w_2))|\le C\big(|w_1-w_2|+|\nabla(w_1-w_2)|\big).$$
By \eqref{eq-estimate-Q}, we obtain, for any $t\ge t_0$, 
\begin{align*} 
&|P(w_1)-P(w_2)|+|\nabla(P(w_1)-P(w_2))|\\
&\qquad\le C(\epsilon(t)+Be^{-\mu t})\big(|w_1-w_2|+|\nabla(w_1-w_2)|\big),
\end{align*}
and hence 
\begin{align*}
&\|P(w_1)-P(w_2)\|_{{C}^{1}_{\mu}([t_0,\infty)\times\mathbb{S}^{n-1})}\\
&\qquad\le C(\epsilon(t_0)+Be^{-\mu t_0})\|w_1-w_2\|_{{C}^{1}_{\mu}([t_0,\infty)\times\mathbb{S}^{n-1})}.
\end{align*}
By Theorem \ref{linthm}, we have 
\begin{align*}
&\|\mathcal{T}(w_1)-\mathcal{T}(w_2)\|_{{C}^{2,\alpha}_{\mu}([t_0,\infty)\times\mathbb{S}^{n-1})}\\
&\qquad \leq C\|P(w_1)-P(w_2)\|_{{C}^{0,\alpha}_{\mu}([t_0,\infty)\times\mathbb{S}^{n-1})}\\
&\qquad\le C(\epsilon(t_0)+Be^{-\mu t_0})\|w_1-w_2\|_{{C}^{2,\alpha}_{\mu}([t_0,\infty)\times\mathbb{S}^{n-1})}. 
\end{align*}
We obtain \eqref{eq-contraction} by choosing $t_0$ sufficiently large. 

{\it Step 4.} We now finish the proof. 
By the contraction mapping principle, we hence have  $w\in{C}^{2,\alpha}_{\mu}([t_0,\infty)\times\mathbb{S}^{n-1})$ satisfying
$\mathcal{T}(w)=w.$ 
This yields  a solution $w\in{C}^{2,\alpha}_{\mu}([t_0,\infty)\times\mathbb{S}^{n-1})$ of \eqref{eq-equation-w}.
Then, $v=\widehat v+w$ is a solution of \eqref{eq-yamabe-v}. 
\end{proof} 

In view of \eqref{eq-def-v}, 
the assumptions \eqref{eq-assumption-leading-u}
and \eqref{eq-assumption-approximate-u}
are equivalent to \eqref{eq-assumption-leading-v}
and \eqref{eq-assumption-approximate-v},  respectively. 
Hence, Theorem \ref{mainthm} follows from Theorem \ref{mainthm-v}. 

\smallskip 

The function $\widehat v$ satisfying \eqref{eq-assumption-leading-v}
and \eqref{eq-assumption-approximate-v} will be called 
an {\it approximate solution} of \eqref{eq-yamabe-v} of order $\mu$ with the leading term $\psi$.

To end this section, 
we provide a general procedure to construct approximate solutions
by perturbing solutions of the linearized equation. 
We will prove that we can always find approximate solutions with a designated order by 
a perturbation of solutions of the linearized equations which decay at infinity. 

We  need the following result. 

\begin{lemma}\label{lemma-SphericalHarmonics} Let $Y_k$ and $Y_l$ be spherical harmonics of 
degree $k$ and $l$, respectively. Then, 
$$Y_kY_l=\sum_{i=0}^{k+l}Z_i,$$
where $Z_i$ is some spherical harmonic of degree $i$, for $i=0,1, \cdots, k+l$. 
\end{lemma}

Lemma \ref{lemma-SphericalHarmonics} follows easily from a well-known 
decomposition of homogeneous polynomials into a finite linear combination of 
homogeneous harmonic polynomials. (Refer to \cite{Stein1971}.)

\begin{prop}\label{prop-approximate-solution}
Let $\psi$ be a positive periodic solution of \eqref{eq-psi}, $\mathcal I$ the index set associated with $\psi$, 
and $\mu>1$ with $\mu\notin \mathcal I$. Suppose that $\eta$ is a solution of $L\eta=0$ on 
$\mathbb R\times \mathbb S^{n-1}$, with $\eta(t,\cdot)\to 0$ as $t\to\infty$ uniformly on $\mathbb S^{n-1}$. 
Then for some $t_0>0$, there exists a smooth function 
$\widetilde \eta$ on $[t_0,\infty)\times \mathbb S^{n-1}$ such that 
$\widehat v=\psi+\eta+\widetilde \eta$ satisfies 
\eqref{eq-assumption-leading-v}
and \eqref{eq-assumption-approximate-v}.\end{prop} 

\begin{proof} In the proof below, we adopt the notation $f=O(g)$ if $|f|\le Cg$. 
We first decompose the index set $\mathcal I$. Set 
$$\mathcal I_\rho=\{\rho_j:\, j\ge 1\},$$ 
and 
$$\mathcal I_{\widetilde \rho}=\Big\{\sum_{i=1}^kn_i\rho_i:\, n_i\in \mathbb Z_+, \sum_{i=1}^kn_i\ge 2\Big\}.$$
We assume $\mathcal I_{\widetilde \rho}$ is given by a strictly increasing sequence 
$\{\widetilde\rho_i\}_{i\ge 1}$, with 
$\widetilde \rho_1=2$. It is possible that $\mathcal I$ and $ \mathcal I_{\widetilde \rho}$ have common elements. 

For each $\widetilde \rho_i\in \mathcal I_{\widetilde \rho}$, 
we consider 
nonnegative integers $n_1, \cdots, n_{k_1}$ such that 
\begin{equation}\label{eq-requirement-m}
n_1+\cdots+n_{k_1}\ge 2, \quad n_1\rho_1+\cdots+n_{k_1}\rho_{k_1}= \widetilde \rho_i.\end{equation}
There are only finitely many collections of 
nonnegative integers $n_1$, $\cdots$, $n_{k_1}$ satisfying \eqref{eq-requirement-m}. 
Set 
\begin{align*}
\widetilde K_i&=\max\{n_1+2n_2+\cdots+k_1n_{k_1}:\\
&\qquad \qquad n_1, \cdots, n_{k_1}\text{ are nonnegative integers 
satisfying \eqref{eq-requirement-m}}\},
\end{align*}
and 
\begin{align}\label{eq-def-M}\widetilde M_i=\max\{m:\, \mathrm{deg}(X_m)\le \widetilde K_i\}.\end{align}
In the following, we only consider the case that $\psi$ is a positive nonconstant period function. 
The case that $\psi$ is a positive constant is easier. 

Take a function $\phi$ such that $|\phi|<\psi$ on $\mathbb R\times \mathbb S^{n-1}$. Then, a simple computation yields 
\begin{align*}
\mathcal N(\psi+\phi)&=L\phi
+\frac14n(n-2)\big[(\psi+\phi)^{\frac{n+2}{n-2}}-\psi^{\frac{n+2}{n-2}}-\frac{n+2}{n-2}\psi^{\frac{4}{n-2}}\phi\big]\\
&=L\phi
+\frac14n(n-2)\psi^{\frac{n+2}{n-2}}\big[(1+\psi^{-1}\phi)^{\frac{n+2}{n-2}}-1-\frac{n+2}{n-2}\psi^{-1}\phi\big].
\end{align*}
Consider the Taylor expansion, for $|s|<1$,  
$$(1+s)^{\frac{n+2}{n-2}}=\sum_{k=0}^\infty a_ks^k,$$ 
where $a_k$ is a constant. 
Here, we write 
the infinite sum just for convenience. We do not need the convergence of 
the infinite series and we always expand up to finite orders. 
Then, by renaming the constant $a_k$, we have 
\begin{align}\label{eq-expansion-N}\mathcal N(\psi+\phi)=L\phi
+\sum_{k=2}^\infty a_k\psi^{\frac{n+2}{n-2}-k}\phi^k.
\end{align}
We point out that the summation above starts from $k=2$.

Let  $I$ be the largest integer such that $\rho_I+\rho_1<\mu$ and 
$\widetilde{I}$ be the largest integer such that $\widetilde\rho_{\widetilde I}<\mu$.  
By Lemma \ref{lemma-Asymptotics-U2a}, Ker$(L_0)$ has no functions converging to 0 as $t\to \infty$
and for $i\ge 1$, Ker$(L_i)$ has one function 
$\psi_i^+(t)=e^{-\rho_i t}p^+_i(t)$  decaying  
exponentially and one function $\psi_i^-(t)=e^{-\rho_i t}p^-_i(t)$ growing exponentially. 
Without loss of generality, 
we may assume that the given solution $\eta$ of $L\eta=0$ has the form
\begin{equation}\label{eq-approximation1}\eta(t,\theta)
=\sum_{i=1}^Ic_ip^+_i(t)X_i(\theta)e^{-\rho_it},\end{equation}
where $c_i$ is a constant. 
This is because that any term of the form $e^{-\rho_it}$ with $i>I$ in $\eta$ 
will contribute terms only of the form $e^{-\widetilde{\rho}_it}$ with $\widetilde{\rho}_i>\mu$ in $\mathcal N(\psi+\eta)$.

We first consider the case that 
\begin{equation}\label{eq-special}\mathcal I_\rho\cap \mathcal I_{\widetilde \rho}=\emptyset.\end{equation}
In other words, no $\rho_i$ can be written 
as a linear combination of some of $\rho_1, \cdots, \rho_{i-1}$ with positive integer coefficients, except a single 
$\rho_{i'}$ which is equal to $\rho_i$. 

We will prove that we can find 
$\widetilde{\eta}_0, \widetilde{\eta}_1, \cdots, \widetilde{\eta}_{\widetilde I}$  successively such that,  
for any $k=0, 1, \cdots, \widetilde I$,  
\begin{equation}\label{eq-approximate-induction-i}
\mathcal N(\psi+\eta+\widetilde{\eta}_0+ \cdots+\widetilde{\eta}_k)=O(e^{-\widetilde{\rho}_{k+1} t}).\end{equation}

We first take $\phi$ to be $\eta$ as in \eqref{eq-approximation1}. By \eqref{eq-expansion-N} and $L\eta=0$, we have 
\begin{align*}\mathcal N(\psi+\eta)=\sum_{n_1+\cdots+n_{k_1}\ge 2}a_{n_1\cdots n_{k_1}}(t)
e^{-(n_1\rho_1+\cdots+n_{k_1}\rho_{k_1})t}X_1^{n_1}\cdots X_{k_1}^{n_{k_1}},\end{align*}
where $n_1, \cdots, n_{k_1}$ are nonnegative integers, and $a_{n_1\cdots n_{k_1}}$ 
is a smooth periodic function. By the definition of $\mathcal I_{\widetilde \rho}$, 
$n_1\rho_1+\cdots+n_{k_1}\rho_{k_1}$ is some 
$\widetilde \rho_i$. 
Hence, by Lemma \ref{lemma-SphericalHarmonics}, 
\begin{equation}\label{eq-property-I}
\mathcal N(\psi+\eta)=\sum_{i=1}^{\widetilde{I}}
\Big\{\sum_{m=0}^{\widetilde M_i}
a_{im}(t)X_{m}(\theta)\Big\}e^{-\widetilde \rho_it}+O(e^{-\widetilde \rho_{\widetilde I+1}t}),\end{equation}
where $\widetilde M_k$ is defined in \eqref{eq-def-M}, and $a_{im}$ is a smooth periodic function. In particular,
$$\mathcal N(\psi+\eta)=O(e^{-\widetilde \rho_1t}).$$
Hence, \eqref{eq-approximate-induction-i} holds for $k=0$ with $\widetilde{\eta}_0=0$. 

Suppose $\widetilde{\eta}_0, \widetilde{\eta}_1, \cdots, \widetilde{\eta}_{k-1}$ are already constructed so that 
\eqref{eq-approximate-induction-i} holds for $0, 1, \cdots, k-1$. We now consider $k$. Let $\widetilde{\eta}_{k}$
be a function of the form 
\begin{equation}\label{eq-Asymptotic-Uk-22}
\widetilde\eta_{k}(t,\theta)=
\Big(\sum_{m=0}^{\widetilde M_k} 
c_{km}(t)X_{m}(\theta)\Big)e^{-\widetilde \rho_kt},
\end{equation}
where $c_{km}$ is a smooth periodic function to be determined. 
A computation similar as that leading to \eqref{eq-property-I} yields 
\begin{align*}\mathcal N(\psi+\eta+\widetilde{\eta}_0+ \cdots+\widetilde{\eta}_k)&=
L\widetilde{\eta}_1+ \cdots+L\widetilde{\eta}_k\\
&\qquad+\sum_{i=1}^{\widetilde I}
\Big\{\sum_{m=0}^{\widetilde M_i}
a_{im}(t)X_{m}(\theta)\Big\}e^{-\widetilde \rho_it}+O(e^{-\widetilde \rho_{\widetilde I+1}t}),\end{align*}
where $a_{im}$ is a smooth periodic function, which may be different from that in \eqref{eq-property-I}. 
By the induction hypothesis, we have 
$$\mathcal N(\psi+\eta+\widetilde{\eta}_0+ \cdots+\widetilde{\eta}_k)=
L\widetilde{\eta}_k+\sum_{i=k}^{\widetilde I}
\Big\{\sum_{m=0}^{\widetilde M_i}
a_{im}(t)X_{m}(\theta)\Big\}e^{-\widetilde \rho_it}+O(e^{-\widetilde \rho_{\widetilde I+1}t}).$$ 
We point out that $\widetilde{\eta}_k$ does not contribute to the coefficient $a_{km}$. We 
take $\widetilde{\eta}_k$ such that 
$$L\widetilde{\eta}_k=-
\Big\{\sum_{m=0}^{\widetilde M_k}
a_{km}(t)X_{m}(\theta)\Big\}e^{-\widetilde \rho_kt}.$$ 
Then, we have \eqref{eq-approximate-induction-i} for $k$. With $\widetilde{\eta}_k$ in \eqref{eq-Asymptotic-Uk-22}, 
it suffices to solve, for $m=0, 1, \cdots, \widetilde{M}_k$,  
\begin{equation}\label{eq-solving-linear1}L_m(c_{km}(t)e^{-\widetilde \rho_kt})=a_{km}(t)e^{-\widetilde \rho_kt}.\end{equation}
Since $\rho_m\neq\widetilde\rho_k$ for any $m$ and $k$, we can find a periodic solution $c_{km}$ of 
\eqref{eq-solving-linear1}. 
In fact, there is an explicit formula for $c_{km}(t)e^{-\widetilde \rho_kt}$ in terms of $a_{km}(t)e^{-\widetilde \rho_kt}$ 
with help of the basis of the kernel Ker$(L_m)$ as in 
Lemma \ref{lemma-Asymptotics-U2a}. 
If $0<\rho_m<\widetilde{\rho}_k$, such an expression is given by \eqref{eq-expression-w-i}. 
If $\rho_m>\widetilde{\rho}_k$, we simply replace the first integral in \eqref{eq-expression-w-i} with 
the one from $t_0$ to $t$. 
If $m=0$, such an expression is given by \eqref{eq-expression-w-0}. 
This finishes the induction. In conclusion, we set 
\begin{equation}\label{eq-approximation2}
\widetilde\eta(t,\theta)=\sum_{i=1}^{\widetilde I}
\Big\{\sum_{m=0}^{\widetilde{M}_i}
c_{im}(t)X_{m}(\theta)\Big\}e^{-\widetilde\rho_it}.\end{equation}
Then, 
$$\mathcal N(\psi+\eta+\widetilde \eta)=O(e^{-\widetilde \rho_{\widetilde I+1}t})=O(e^{-\mu t}).$$
We have a similar estimate for the gradient of $\mathcal N(\psi+\eta+\widetilde \eta)$. 
This finishes the proof for the case \eqref{eq-special}. 
 
Next, we consider the general case; namely, some $\rho_i$ can be written 
as a linear combination of some of $\rho_1, \cdots, \rho_{i-1}$ with positive integer coefficients. 
We indicate how to modify discussion above to treat the general case. 
The modification mainly concerns \eqref{eq-solving-linear1}. If $\rho_m=\widetilde{\rho}_k$, then for a periodic function 
$a_{km}$, instead of \eqref{eq-solving-linear1}, we can find periodic functions 
$c_{k0m}$ and $c_{k1m}$ such that 
$$L_m\big((c_{k0m}(t)+tc_{k1m})e^{-\widetilde \rho_kt}\big)=a_{km}(t)e^{-\widetilde \rho_kt}.$$
Note that an extra power of $t$ appears when solving the above ordinary differential equation. 
Such a power of $t$ will generate more powers of $t$ upon iteration. In general, for periodic functions 
$a_{kjm}$ with $j=0, \cdots, J$ for some nonnegative integer $J$, 
we can find periodic functions $c_{kjm}$ with $j=0, \cdots, J+1$ such that 
$$L_m\Big(\sum_{j=0}^{J+1}c_{kjm}(t)t^je^{-\widetilde \rho_kt}\Big)
=\sum_{j=0}^{J}a_{kjm}(t)t^je^{-\widetilde \rho_kt}.$$
In conclusion, instead of \eqref{eq-approximation2}, we will take 
\begin{equation}\label{eq-approximation3}
\widetilde\eta(t,\theta)=\sum_{i=1}^{\widetilde I}\sum_{j=0}^{i}
\Big\{\sum_{m=0}^{\widetilde{M}_i}
c_{ijm}(t)X_{m}(\theta)\Big\}t^je^{-\widetilde\rho_it},\end{equation}
where $c_{ijm}$ is a smooth periodic function. 
\end{proof} 

In \cite{HLL201?}, we proved that any positive solution $u$ of \eqref{eq-Yamabe} in $B_1\setminus\{0\}$,
with a nonremovable singularity at the origin, satisfies 
\eqref{eq-estimate-u-k}. In fact, an obviously revised version of \eqref{eq-estimate-u-k} holds 
for approximate solutions of appropriate order. 
The proof of Proposition \ref{prop-approximate-solution} 
above is the converse procedure used in the proof of \eqref{eq-estimate-u-k} in \cite{HLL201?}.

\end{document}